\documentclass[a4paper,11pt]{article}

\usepackage{a4wide}
\usepackage{amsmath,amsthm}
\usepackage{xspace}
\usepackage{pifont}
\usepackage{graphicx}
\usepackage{amssymb}
\usepackage{epic, eepic}
\usepackage{dsfont}
\usepackage{amssymb}
\usepackage{makeidx}
\usepackage{mathrsfs}
\usepackage{exscale}
\usepackage{color} 
\usepackage{overpic} 
\usepackage{bm}
\usepackage{bbm}

\usepackage{amsmath,afterpage}
\usepackage{epsf}
\usepackage{graphics,color}

\def\0{\mathbf{0}}

\def\eps{\varepsilon}

\def\rr{\rightarrow}

\def \< {\langle}
\def \> {\rangle}

\def\beqa{\begin{eqnarray}}
\def\eeqa{\end{eqnarray}}
\def\beqas{\begin{eqnarray*}}
\def\eeqas{\end{eqnarray*}}

\newtheorem{theorem}{Theorem}[section]
\newtheorem{lemma}[theorem]{Lemma}

\newtheorem{corollary}[theorem]{Corollary}

\numberwithin{equation}{section}
\newcommand{\hatd}[1]{{}}

\newcommand{\bd}{\begin{displaymath}}
\newcommand{\ed}{\end{displaymath}}
\newcommand{\be}{\begin{equation}}
\newcommand{\ee}{\end{equation}}
\newcommand{\bq}{\begin{eqnarray}}
\newcommand{\eq}{\end{eqnarray}}
\newcommand{\bn}{\begin{eqnarray*}}
\newcommand{\en}{\end{eqnarray*}}
\newcommand{\dl}{\delta}
\newcommand{\re}{\mathds{R}}

\def\wt{\widetilde}

\begin{document}

\title{Fractional Brownian motion with zero Hurst parameter:\\
a rough volatility viewpoint}
\author{Eyal Neuman\thanks{Imperial College London, Department of Mathematics, e.neumann@imperial.ac.uk}, Mathieu Rosenbaum\thanks{\'Ecole Polytechnique, CMAP, mathieu.rosenbaum@polytechnique.edu}}
\maketitle

\abstract \noindent Rough volatility models are becoming increasingly popular in quantitative finance. In this framework,
one considers that the behavior of the log-volatility process of a financial asset
is close to that of a fractional Brownian motion with Hurst parameter around 0.1. Motivated by this, we wish to define a natural and relevant limit for the fractional Brownian motion when $H$ goes to zero. We show that once properly normalized, the fractional Brownian motion converges to a Gaussian random distribution which is very close to a log-correlated random field.\\

\bigskip

\parindent=0mm

\section{Introduction}   \label{section-intro} 

The fractional Brownian motion (fBm for short) is a very popular modeling object in many fields such
as hydrology, see for example \cite{molz1997fractional}, telecommunications and network
traffic, see  \cite{leland1994self,mikosch2002network} among others and finance, see the seminal paper \cite{comte1998long}. A fBm $(B_t^H)_{t \in \mathbb{R}}$ with Hurst parameter $H \in (0,1)$
is a zero-mean Gaussian process with covariance kernel given by
\[
\mathbb{E}[B_t^H B_s^H]= \frac{1}{2} \left(|t|^{2H} + |s|^{2H} -|t-s|^{2H} \right).
\] 

It has stationary increments and is self-similar with parameter $H$, that is
$(B_{at}^H)_{t \in\mathbb{R}}$ has the same law as $(a^HB_{t}^H)_{t \in\mathbb{R}}$ for any $a>0$. Furthermore, sample paths of fBm have almost surely
H\"older regularity $H-\varepsilon$ for any $\varepsilon>0$.\\

One of the main probabilistic features which motivates the use of fBm in the applications mentioned above, is
the long memory property of the increments when $H>1/2$. This means that for $H>1/2$, we have
$$\sum_{i=1}^{+\infty}\text{Cov}[(B_{i+1}^H-B_i^H),B_1^H]=+\infty.$$ Thus the auto-covariance function of the fBm increments decays slowly, which is interesting when modeling persistent phenomena.\\ 

However, recently, a new paradigm has been introduced in \cite{Gat-Jai-Ros14} for the use of fBm in finance. Indeed, a careful analysis of financial time series suggests that the log-volatility process, 
that is the intensity of the price variations of an asset, actually behaves like a fBm with Hurst parameter of order $0.1$. Hence various approaches using a fBm with small Hurst parameter have been introduced for volatility modeling. These models are referred to as {\it rough volatility models}, see \cite{bayer2017regularity,bayer2016pricing,bayer2017short, bennedsen2017hybrid,el2016characteristic,forde2015asymptotics,fukasawa2017short,jacquier2017pathwise} for more details and practical applications.\\

Such small estimated values for $H$ (between 0.05 and 0.2) have been found when studying the volatility process of thousands of assets, 
see \cite{bennedsen2016decoupling}. Consequently, a natural question is the behavior of the fBm in the limiting case when $H$ is sent to zero. Of course, putting
directly $H=0$ in the covariance function does not lead to a relevant process. Thus, in this work, we wish to 
build a suitable sequence of normalized fBms and describe its limit as $H$ goes to zero. This will lead us to a possible definition of the fractional Brownian motion for $H=0$. Note that several authors already defined some fractional Brownian motion for $H=0$, see in particular \cite{FKS16}. This is usually done through a regularization procedure. Our approach here is quite simple and probably more natural from the financial viewpoint we have in mind. Instead of regularizing the process, we choose to normalize it in order to get a non-degenerate limit. Our normalized sequence of processes $(X^{H}_{.})_{H\in(0,1)}$ is defined through
$$
X^{H}_{t}=\frac{B_{t}^{H}-\frac{1}{t}\int_{0}^{t}B^{H}_{u}\,du}{\sqrt{H}}, \quad  t\in\mathbb{R},
$$
where $X_{0}^{H}=0$. Substracting the integral in the numerator and dividing by $\sqrt{H}$ enables us to get a non-trivial limit for our sequence as $H$ goes to $0$.\\

\noindent Our main result is the convergence of $X^{H}$, seen as a random element in the space of tempered distributions, towards an approximately {\it log-correlated Gaussian field}. Denote by $\mathcal S$ the real Schwartz space, that is the set of real-valued functions on $\re$ whose derivatives of all orders exist and decay faster than any polynomial at infinity. We write $\mathcal S'$ for the dual of $\mathcal S$, that is the space of tempered distributions. We also define the subspace $\mathcal S_{0}$ of the real Schwartz space, consisting of functions $\phi$ from $\mathcal S$ with $\int_{\re}\phi(s)\,ds=0$, and its topological dual $\mathcal S' / \re$. A log-correlated Gaussian field (LGF for short) $X\in \mathcal S'/ \re$, is a centered Gaussian field whose covariance kernel satisfies 
$$
\mathbb{E}[\langle X^{}_{}, \phi_{1} \rangle^{}\langle X^{}_{}, \phi_{2} \rangle^{}] =  \int_{\re}\int_{\re} \log \frac{1}{|t-s|} \phi_{1}(t)\phi_{2}(s)\,dt\,ds, 
$$
for any $\phi_{1},\phi_{2}\in \mathcal S_{0}$, see \cite{Dup-Rho-She-Var2014} for an overview on LGF. We show in this paper that the limit of $X^{H}$ as $H$ goes to zero is ``almost" a log-correlated Gaussian field, see Section \ref{section-results} for an accurate result.\\

LGFs are closely related to some multifractal processes pioneered by Mandelbrot (see for example \cite{Man-Cal-Fis97}), and further developed in \cite{bacry2001multifractal,Muz-Bac03, Bar-Men02, cal-fis04}, among others. A process $(Y_t)_{t\geq 0}$ is said to be multifractal if for a range of values of $q$, we have for some $T>0$ 
$$
\mathbb{E}\big[|Y_{t+\ell}-Y_t|^{q}\big] \sim C(q)\ell^{\zeta(q)}, \quad \textrm{for } 0<\ell \leq T, 
$$
where $C(q)>0$ is a constant and $\zeta(\cdot)$ is a non-linear concave function. In particular, the multifractal random walk model for the log-price of an asset  in \cite{bacry2001multifractal} satisfies such property. It is defined as $Y_t=B_{M([0,t])},$
where $B$ is a Brownian motion and $$M(t)=\underset{l\rightarrow 0}{\text{lim }}\sigma^2\int_0^t e^{w_l(u)}du, \text{ a.s.},$$ with $\sigma^2>0$ and
$w_l$ a Gaussian process such that for some $\lambda^2>0$ and $T>0$ 
$$\text{Cov}[w_l(t),w_l(t')]=\lambda^2\text{log}(T/|t-t'|), \text{ for }l<|t-t'|\leq T,$$   
see \cite{Muz-Bac03} for details. Hence we see that $M$ formally corresponds to a measure of the form $\text{exp}(X_t)dt$, where $X$ is a LGF. For the precise definition of such measures, see \cite{kahane85} and the generalizations on Gaussian multiplicative chaos in \cite{Rhodes-Vargas14,vargas16,Robert-Vargas10} and the references therein. Finally note that LGFs and more generally the associated theory of Gaussian multiplicative chaos have extensive use in other fields than finance, such as turbulence, see \cite{Che-Rob-Var10,Fyo-Dou-Ros10}, disordered systems, see \cite{FyoBcu08,madaule2016} and Liouville quantum gravity, see \cite{Rhodes-Vargas14,vargas16}.\\ 

In the following section we introduce our main theorem, that is an accurate statement about the convergence of the normalized fBm towards a LGF as $H$ goes to zero. We also discuss the multifractal properties of the limiting LGF in the same section. The proof of our theorem can be found in Section \ref{theorem-proof}.  
 
\section{Convergence of the fBM towards a LGF} \label{section-results}
\subsection{Main result}
We define the weak convergence of elements in $\mathcal S'$ as in Proposition 12.2 in \cite{Lod-She-Sun-Wat16}. 
We say that $X^{H}$ converges weakly to $X$ as $H$ tends to $0$ if for any $\phi \in \mathcal S$ we have 
\bn
\langle X^{H}, \phi \rangle \rr \langle X^{}, \phi \rangle, 
\en
in law, as $H$ tends to $0$. The main result of our paper is the following.
\begin{theorem}  \label{theorem-conv} 
The sequence $\{X_{t}^{H}\}_{t \in \re}$ converges weakly as $H$ tends to zero towards a centered Gaussian field $X$ satisfying for any $\phi_{1},\phi_{2}\in \mathcal S$
$$\mathbb{E}[\langle X^{}_{}, \phi_{1} \rangle^{}\langle X^{}_{}, \phi_{2} \rangle^{}] =  \int_{\re}\int_{\re} K(t,s) \phi_{1}(t)\phi_{2}(s)\,dt\,ds, 
$$
where for $-\infty<s,t<\infty$, $s\neq t$ and $s,t\neq 0$ $$
 K(t,s)=\log\frac{1}{|t-s|} +g(t,s), 
$$
with 
$$
g(t,s)=\frac{1}{t}\int_{0}^{t}\log|s-u|du+\frac{1}{s}\int_{0}^{s}\log|t-u|du-\frac{1}{ts}\int_{0}^{t}\int_{0}^{s}\log|u-v|du dv.$$
 \end{theorem}

We see that when $t,s>\dl$ for some $\dl>0$, then $g(t,s)$ is a bounded continuous function. Hence the covariance kernel exhibits the same type of singularity as that of a LGF. Consequently, in our framework, the limit when $H$ goes to zero of a normalized version of the fBm is ``almost" a LGF.

\subsection{Multifractal properties}

In \cite{Rhodes-Vargas14}, the authors study the case of centered Gaussian fields on any domain $D \subset \re $ with  covariance kernel satisfying
\be  \label{var-ker}
K(x,y)= \log_{+}\frac{1}{|x-y|}+f(x,y),  \quad x,y\in D, 
\ee
where $\log_{+}(x) :=\max\{0,\log x\}$ and $f(x,y)$ is a bounded continuous function. Thus, if we restrict $X$ to the domain $[\dl,1]$ for some fixed $\dl>0$, then $X$ is included in the framework of \cite{Rhodes-Vargas14}. In particular, their results on the multifractal spectrum of Gaussian fields apply on this restricted domain. \\

%We first define a smooth approximation of $X$. Let $\phi:\re \rr \re$ be a $\mathcal C^{1}$ function which is supported on the ball $B(0,1)$ and satisfies $\int_{\re}\phi(s)ds =1$. 
%Let $\eps>0$. For $t,s \in \re$ we set $\phi^{\eps}_{t}(s) = \frac{1}{\eps}\phi((t-s)/\eps)$. 
%The $\eps$--mollified version of $X$ is given by  
%\bn
%X_{\eps}(t)= \langle X^{}_{}, \phi^{\eps}_{t} \rangle^{}.
%\en
Motivated by the preceding paragraph, we first fix $\dl>0$ and define an approximate volatility measure $\xi_{\gamma}^{H}$ by
\bn
\xi_{\gamma}^{H}(dt) = e^{\gamma X^{H}_{t}-\frac{\gamma^{2} }{2}\mathbb{E}[(X_{t}^{H})^{2}]}dt,  \quad \dl \leq t\leq 1, 
\en
for some constant $\gamma>0$. Here we assume that $\xi_{\gamma}^{H}(\cdot)$ vanishes on $[\dl,1]^{c}$. 
From the result of Theorem \ref{theorem-conv} we deduce the following corollary. In what follows, convergence in the $L^{1}$ norm stands for the usual convergence of random variables in $L^{1}$.
\begin{corollary} \label{corollary-gmc} 
For $\gamma < \sqrt{2}$, $\{\xi_{\gamma}^{H}\}_{H\in (0,1)}$ converges as $H$ approaches zero to a random measure $\xi_{\gamma}$ in the following sense, 
$$
 \int_{\re}\phi(t)\xi_{\gamma}^{H}(dt) \stackrel{L^{1}}{\rightarrow}  \int_{\re}\phi(t)\xi_{\gamma}(dt),\quad \textrm{for all } \phi \in \mathcal S. 
$$ 
Moreover, the limiting measure $\xi_{\gamma}$ is the so-called Gaussian multiplicative chaos. 
\end{corollary}  

The proof Corollary \ref{corollary-gmc} is given in Section \ref{theorem-proof}.
For the definition and properties of Gaussian multiplicative chaos we refer to a survey paper by Rhodes and Vargas \cite{Rhodes-Vargas14}. We will briefly explain some of the properties of $\xi_{\gamma}$ from the theory of Gaussian multiplicative chaos.  \\ \\
We first describe the behavior of the moments of $\xi_{\gamma}$. Note that for $t\in [\dl , 1]$, $K$ is of the form (\ref{var-ker}).  From Proposition 2.5 in \cite{vargas16}, it follows that for all $t\in (\dl , 1)$ and $q\in (-\infty, 2/\gamma^{2})$, there exists $C(t,q)>0$ such that 
$$ 
\mathbb{E}\big[\xi_{\gamma}\big(B(t,r)\big)^{q}] \sim C(t,q)r^{\zeta(q)},  \quad \textrm{as } r\rr0, 
$$
where $\zeta(q) = (1+\gamma^{2}/2)q - \gamma^{2}q^{2}/2$. Hence we do obtain a multifractal scaling as described in the introduction.\\

Finally we consider a quantity closely related to the function $\zeta$ above: the spectrum of singularities of $\xi_{\gamma}$. For any  $0<\gamma < \sqrt{2}$ and $0<r<\sqrt{2}/{\gamma}$, we define 
\bn
 G_{\gamma,r} =\big\{ x\in (\dl,1); \ \lim_{\eps\rr 0 }\frac{\log \xi_{\gamma}(B(x,\eps))}{\log \eps} =1+\big(\frac{1}{2}-r\big)\gamma^{2} \big \}. 
\en

The set $G_{\gamma,r}$ somehow corresponds to the points $x$ where the H\"older regularity of $\xi_{\gamma}$ is equal to $1+(1/2-r)\gamma^{2}$.
Let $dim_{H}(A)$ denote the Hausdorff dimension of a set $A$. Theorem 2.6 in \cite{vargas16} states that  
\bn
dim_{H}(G_{\gamma,r}) = 1-\frac{\gamma^{2}r^{2}}{2}.
\en
In particular, we remark that
$$dim_{H}(G_{\gamma,r})=\underset{p\in \mathbb{R}}{\text{inf}}\Big\{p\big(1+\big(\frac{1}{2}-r\big)\gamma^{2}\big)-\zeta(p)+1\Big\}.$$
This equality means that the Frish-Parisi conjecture relating the scaling exponents of a process to its spectrum of singularities holds in our case, see \cite{frisch1985singularity} and \cite{Rhodes-Vargas14} for more details on the multifractal formalism.\\

%Another interesting property of the volatility process $\xi$, is the that is lives on $G_{\gamma,r}$. This property is stated as follows (see Theorem 4.1 of  %\cite{Rhodes-Vargas14}) 
%\bn
%\xi_{r\gamma}(  G_{\gamma,r}^{c}) = 0, \q  \textrm{a.s.},  \ \textrm{for all } 0<r<\sqrt{2}/{\gamma}, \ 0<\gamma  < \sqrt{2}.
%\en

The rest of this paper is devoted to the proofs of Theorem \ref{theorem-conv} and Corollary \ref{corollary-gmc}. 

\section{Proof of Theorem \ref{theorem-conv} and Corollary \ref{corollary-gmc}.}\label{theorem-proof} 
For $t,s\in\mathbb{R}$, let $K_{H}(t,s)=\mathbb{E}[X^H_tX^H_s]$ . We start the proof with the following important lemma. 
\begin{lemma}  \label{lemma-con}
For any non-zero $-\infty <s, t <\infty$ with $s\neq t$, we have 
 \be \label{K-con}
\lim_{H\rr 0} K_{H}(t,s) 
= K(t,s). 
\ee
\end{lemma} 
\begin{proof} We write
$K_{H}(t,s)=I_{1}^{H}(t,s)+ I_{2}^{H}(t,s)+I_{3}^{H}(t,s)+I_{4}^{H}(t,s)$,
where
\be \label{KH-deomp}
\begin{aligned} 
I_{1}^{H}(t,s)&=-\frac{1}{2H}|t-s|^{2H},~~I_{2}^{H}(t,s)=\frac{1}{2H}\frac{1}{t}\int_{0}^{t}|s-u|^{2H}du, \\
I_{3}^{H}(t,s)&=\frac{1}{2H}\frac{1}{s}\int_{0}^{s}|t-u|^{2H}du,~~I_{4}^{H}(t,s)=-\frac{1}{2H}\frac{1}{st}\int_{0}^{t}\int_{0}^{s}|u-v|^{2H}dudv 
\end{aligned} 
\ee
and $K(t,s)=I_{1}^{0}(t,s)+ I_{2}^{0}(t,s)+I_{3}^{0}(t,s)+I_{4}^{0}(t,s)$,
where
$$I_{1}^{0}(t,s)=\log\frac{1}{|t-s|},~~I_{2}^{0}(t,s)=\frac{1}{t}\int_{0}^{t}\log|s-u|du,$$
$$I_{3}^{0}(t,s)=\frac{1}{s}\int_{0}^{s}\log|t-u|du,~~I_{4}^{0}(t,s)=-\frac{1}{st}\int_{0}^{t}\int_{0}^{s}\log|u-v|dudv.$$
We have
$$K_{H}(t,s)= K_{H}(s,t)\text{ and }K(t,s)=  K(s,t),$$ so we can assume that $-\infty <s<t <\infty$.\\

Note that for any $s\not = t$,  
$$
I_{1}^{H}(t,s)+\frac{1}{2H} = -\frac{1}{2H} \big(e^{2H\log|t-s|}  -1\big)
$$
and therefore
$$
\lim_{H \rr 0} I_{1}^{H}(t,s)+\frac{1}{2H}= -\log|t-s|=I_{1}^{0}(t,s).  
$$
Next we deal with $I_{i}^H$, $i=2,...,4$. We consider several cases.\\ 

\textbf{Case 1:} Assume that $0<s<t$. We easily get 
$$
I_{2}^{H}(t,s) +I_{3}^{H}(t,s) = \frac{1}{2H(2H+1)ts}\big(t^{2H+2}+s^{2H+2}-(t-s)^{2H+2}\big). 
$$
For $I_{4}^{H}(t,s)$, note that 
$$
\int_{0}^{t}|u-v|^{2H}\,du=\frac{1}{2H+1} v^{2H+1} +\frac{1}{2H+1} (t-v)^{2H+1}. 
$$
Hence we obtain 
$$
\int_{0}^{s}\int_{0}^{t}|u-v|^{2H}\,du dv= \frac{1}{(2H+1)(2H+2)} \big(t^{2H+2}-(t-s)^{2H+2}+s^{2H+2}\big) 
$$
and therefore
$$  
I_{4}^{H}(t,s) =- \frac{1}{2H(2H+1)(2H+2)ts} \big(t^{2H+2}+s^{2H+2}-(t-s)^{2H+2}\big).
$$ 
Define 
\bn
I^{H}_{2,4}(t,s)=\sum_{i=2}^{4}I_{i}^{H}(t,s). 
\en
It follows that 
\bn
I^{H}_{2,4}(t,s) =  \frac{1}{2H(2H+2)ts} \big(t^{2H+2}+s^{2H+2}-(t-s)^{2H+2}\big).
\en

Thus we have  
\bn
&&I_{2,4}^{H}(t,s) - \frac{1}{H(2H+2)}\\
&&= \frac{1}{2H(2H+2)ts} \big(t^{2H+2}+s^{2H+2}-(t-s)^{2H+2}\big)- \frac{1}{2H(2H+2)ts} \big(t^{2}+s^{2}-(t-s)^{2}\big)  \\
&&= \frac{1}{2H(2H+2)ts} \big(t^{2}(t^{2H}-1)+s^{2}(s^{2H}-1)-(t-s)^{2}((t-s)^{2H}-1)\big). 
\en
Consequently, 
$$\lim_{H \rr 0}I_{2,4}^{H}(t,s)- \frac{1}{H(2H+2)}= \frac{1}{2ts} \big(t^{2}\log t+s^{2}\log s-(t-s)^{2}\log(t-s)\big)$$   
and $$\lim_{H \rr 0}I_{2,4}^{H}(t,s)- \frac{1}{2H}= \frac{1}{2ts} \big(t^{2}\log t+s^{2}\log s -(t-s)^{2}\log(t-s)\big) -\frac{1}{2}.$$

We also easily get that for $0<s<t$ 
$$
I_{2}^{0}(t,s)+I_{3}^{0}(t,s)=\frac{1}{st}\big(t^{2}\log t+ s^{2}\log s -(t-s)^{2}\log(t-s)\big)-2.
$$
Furthermore, 
$$
\int_{0}^{t}\log|u-v|\,du = v \log v +(t-v)\log(t-v) -t. 
$$
Thus we obtain  
\bn
\int_{0}^{s}\int_{0}^{t}\log|u-v|\,du\,dv
&=& -st + \frac{s^{2}}{4}(2\log s-1)+ \frac{t^{2}}{4}(2\log t-1) \\
&&\quad - \frac{(t-s)^{2}}{4}(2\log (t-s)-1) \\
&=& \frac{1}{2}\big(t^{2}\log t +s^{2}\log s -(s-t)^{2}\log(s-t)-3st\big)
\en
and we deduce
\bn
I_{4}^{0}(t,s) =-\frac{1}{2ts}\big(t^{2}\log t +s^{2}\log s -(s-t)^{2}\log(s-t)\big)+\frac{3}{2}. 
\en
Now define 
\bn
I_{2,4}^{0}(t,s)=
\sum_{i=2}^{4}I_{i}^{0}(t,s).
\en
We have 
\bn
I_{2,4}^{0}(t,s)
=\frac{1}{2st}\big(t^{2}\log t+ s^{2}\log s -(t-s)^{2}\log(t-s)\big)-\frac{1}{2}.
\en
Finally we obtain 
$$
 \lim_{H \rr 0}I_{2,4}^{H}(t,s)- \frac{1}{2H} =I_{2,4}^{0}(t,s), 
$$
from which (\ref{K-con}) readily follows for $0<s<t<\infty$. 
\\

\textbf{Case 2:} Assume that $-\infty <s<t <0$. 
We use that 
$$
K_{H}(t,s) = K_{H}(-t,-s)= K_{H}(-s,-t)\text{ and }K(t,s) = K(-t,-s)= K(-s,-t)
$$
to deduce that (\ref{K-con}) follows from the proof of Case 1. \\

\textbf{Case 3:} Assume that $ s<0<t $. We write $s=-u$. 
Repeating the same steps as in Case 1, we get 
\bn
&&I_{2,4}^{H}(t,-u)- \frac{1}{H(2H+2)}  \\
&&= \frac{1}{2H(2H+2)tu} \big((t+u)^{2}((t+u)^{2H}-1)-t^{2}(t^{2H}-1)-u^{2}(u^{2H}-1)\big), 
\en
and 
\bn
I_{2,4}^{0}(t,-u) =\frac{1}{2ut}\big((t+u)^{2}\log(t+u)-t^{2}\log t- u^{2}\log u \big)-\frac{1}{2}.
\en
It follows that 
$$\lim_{H \rr 0}I_{2,4}^{H}(t,-u)- \frac{1}{2H}=I_{2,4}^{0}(t,-u)$$
and we get (\ref{K-con}).\end{proof} 
Let $\phi_{1}, \phi_{2} \in \mathcal S$. Since $X^{H}$ and $X$ are centered Gaussians taking values in $\mathcal S'$, to prove Theorem \ref{theorem-conv}, it is enough to show that
$$
\lim_{H\rr0}\mathbb{E}[\langle X^{H}_{}, \phi_{1} \rangle^{}\langle X^{H}_{}, \phi_{2} \rangle^{}]  = \int_{\re}\int_{\re} K(t,s) \phi_{1}(t)\phi_{2}(s)\,dt\,ds. 
$$
Furthermore, for any $H\in (0,1)$ we have 
$$
\mathbb{E}[\langle X^{H}_{}, \phi_{1} \rangle^{}\langle X^{H}_{}, \phi_{2} \rangle^{}] =\int_{\re}\int_{\re} K_{H}(t,s)\phi_{1}(t)\phi_{2}(s)\,ds \,dt. 
$$
Hence Theorem \ref{theorem-conv} immediately follows from the next lemma. 
\begin{lemma}  \label{lemma-dom-con}
For any $\phi_1,\phi_2 \in \mathcal S$, 
\be \label{ker-L-0}
\lim_{H\rr0}\int_{\re}\int_{\re} K_{H}(t,s)\phi_{1}(t)\phi_{2}(s)\,ds \,dt =\int_{\re}\int_{\re} K_{}(t,s)\phi_{1}(t)\phi_{2}(s)\,ds \,dt.
\ee 
\end{lemma} 
\begin{proof} 
First note that for $x \geq 0$, $1-e^{-x}\leq x$. Therefore for any $0<|t-s|\leq 1$,
$$
\frac{|1-|t-s|^{2H}|}{2H}=\frac{1}{2H}(1-e^{2H \log |t-s|}) \\
\leq- \log |t-s|.$$
Moreover, 
\bn
-\int \int_{|t-s|\leq 1}  \log |t-s||\phi_1(t)||\phi_2(s)| \,dt \,ds &\leq& -\|\phi_2\|_{\infty}\int_{\re} |\phi_1(t)| \int_{|v|\leq 1} \log |v| \,dv \,dt \\
&=&   -2\|\phi_2\|_{\infty}\int_{\re} |\phi_1(t)| \int_{0}^{1} \log(v) \,dv \,dt \\
&\leq & C\|\phi_1\|_{1}\|\phi_2\|_{\infty}.  
\en
Hence, from dominated convergence, it follows that 
$$
\lim_{H\rr 0}\int \int_{|t-s|\leq 1}\frac{1-|t-s|^{2H}}{2H} \phi_1(t)\phi_2(s) \,dt \,ds = -\int \int_{|t-s|\leq 1} \log |t-s|\phi_1(t)\phi_2(s)\,dt \,ds. 
$$
Let $f(x)=x^{2H}/(2H)$, where $x>0$ and $H\in(0,1/2)$. Since $f$ is concave, for any $x\geq 1$, we have $0\leq f(x)-f(1) \leq f'(1)(x-1)$. Therefore for any $|t-s|\geq 1$,
$$
\frac{||t-s|^{2H}-1|}{2H} \leq  |t-s|-1\leq 2(|t|+|s|).$$
Since 
\bn
2\int \int_{|t-s|> 1} (|t|+|s|)|\phi_1(t)||\phi_2(s)| \,dt \,ds  <\infty, 
\en
we get from dominated convergence 
$$ 
\lim_{H\rr 0}\int \int_{|t-s|> 1}\frac{1-|t-s|^{2H}}{2H} \phi_1(t)\phi_2(s) \,dt \,ds = -\int \int_{|t-s|\geq 1} \log |t-s| \phi_1(t)\phi_2(t) \,dt \,ds. 
$$
It remains to show that 
$$
\lim_{H\rr 0}\int_{\re}\int_{\re}\big(I_{2,4}^{H}(t,s)- \frac{1}{2H}\big)\phi_1(t)\phi_2(s)\,dt \,ds=\int_{\re}\int_{\re}I_{2,4}^{0}(t,s)\phi_1(t)\phi_2(s)\,dt \,ds. 
$$
We again consider several cases.\\  

\textbf{Case 1:} Assume that $0<s<t$.
We obviously have 
$$
 \frac{s^{2}}{2H(2H+2)ts} |s^{2H}-1| \leq  \frac{1}{2H(2H+2)}|s^{2H}-1|.
$$
Moreover, from a Taylor expansion, we get $t^{2H}-(t-s)^{2H}\leq 2H(t-s)^{2H-1}s$ and therefore 
\bq  
&& \frac{1}{2H ts}|t^{2}(t^{2H}-1)-(t-s)^{2}((t-s)^{2H}-1)| \nonumber \\
&& \leq   \frac{1}{2Hts} \big( |t^{2}(t^{2H}-1)-(t-s)^{2}(t^{2H}-1)| \nonumber \\
&&\quad +|(t-s)^{2}(t^{2H}-1)-(t-s)^{2}((t-s)^{2H}-1)|\big)\nonumber  \\
&& \leq   \frac{1}{2Hts} \big(|t^{2H}-1||s^{2}-2st|+2H(t-s)^{2}(t-s)^{2H-1}s\big) \nonumber \\
&& \leq   \frac{1}{H}|t^{2H}-1|+(t-s)^{2H}\nonumber . 
\eq
It follows that for any $0< s< t$, 
\be \label{I24-bnd} 
\big|I_{2,4}^{H}(t,s)- \frac{1}{H(2H+2)}\big|\leq \frac{1}{2H}|s^{2H}-1| +   \frac{1}{2H}|t^{2H}-1|+(t-s)^{2H}:=L_{1}^{H}(t,s).
\ee

\textbf{Case 1bis:} Assume that $0<t<s$.
Since $I_{2,4}^{H}(t,s)=I_{2,4}^{H}(s,t)$ we get
$$
\big|I_{2,4}^{H}(t,s)- \frac{1}{H(2H+2)}\big| \leq L_{1}^{H}(s,t).
$$

\textbf{Case 2:} Assume that  $-\infty<-t,-s<0$.
Since $I_{2,4}^{H}(-t,-s) = I_{2,4}^{H}(t,s)$, we have 
$$\big|I_{2,4}^{H}(-t,-s)- \frac{1}{H(2H+2)}\big| \leq L_{1}^{H}(t,s).$$

\textbf{Case 3:} Assume that $ -s<0<t $ and $s\leq t$. 
Recall that when $ -s<0<t $,
\bn
&&I_{2,4}^{H}(t,-s)- \frac{1}{H(2H+2)}  \\
&&= \frac{1}{2H(2H+2)ts} \big((t+s)^{2}((t+s)^{2H}-1)-t^{2}(t^{2H}-1)-s^{2}(s^{2H}-1)\big). 
\en
Since $s<t$ we have 
\bn
\big| \frac{s^{2}}{2H(2H+2)ts}(s^{2H}-1)  \big|\leq \frac{1}{2H(2H+2)}|s^{2H}-1|.
\en
Using a Taylor expansion, we get $(t+s)^{2H}-t^{2H}\leq 2H t^{2H-1}s$ and therefore 
\bn
&& \frac{1}{2H ts}|(t+s)^{2}((t+s)^{2H}-1)-t^{2}(t^{2H}-1)| \nonumber \\
&& \leq   \frac{1}{2Hts} \big(|(t+s)^{2}((t+s)^{2H}-1) -(t+s)^{2}(t^{2H}-1)|  \nonumber \\
&&\qquad \qquad + |(t+s)^{2}(t^{2H}-1)-t^{2}(t^{2H}-1)|\big)\nonumber  \\
&& \leq   \frac{1}{2Hts} \big(2H(t+s)^{2}t^{2H-1}s+|t^{2H}-1||s^{2}+2st|\big) \nonumber \\
&& \leq  (t+s)^{2H}+  \frac{3}{2H}|t^{2H}-1|.
\en
It follows that $$\big|I_{2,4}^{H}(t,-s)- \frac{1}{H(2H+2)}\big|\leq \frac{1}{2H}|s^{2H}-1| +\frac{1}{H}|t^{2H}-1|+(t+s)^{2H}:=L_{2}^{H}(t,s).$$

\textbf{Case 4:} Assume that $ -s<0<t $ and $t < s$. Repeating the same lines as in Case 3 but exchanging the roles of $t$ and $s$ we get 
$$
\Big|I_{2,4}^{H}(t,-s)- \frac{1}{H(2H+2)}\Big| \leq L_{2}^{H}(s,t).\\
$$

\textbf{Case 5:} Assume that $ -t<0<s $ and $t< s$. 
Since $I_{2,4}^{H}(-t,s) = I_{2,4}^{H}(s,-t)$, it follows from Case 3 that 
\bn
\Big|I_{2,4}^{H}(-t,s)- \frac{1}{H(2H+2)}\Big| \leq L_{2}^{H}(s,t).\\
\en

\textbf{Case 6:} Assume that $ -t<0<s $ and $s <t$. 
Using again $I_{2,4}^{H}(-t,s) = I_{2,4}^{H}(s,-t)$, we get from Case 4 that
\bn
\Big|I_{2,4}^{H}(-t,s)- \frac{1}{H(2H+2)}\Big| \leq L_{2}^{H}(t,s).
\en

Finally, we clearly have 
$$
\int_{\re}\int_{\re}\Big(\sup_{H\in (0,1/2)}\big(L_{1}^{H}(s,t)+L_{1}^{H}(t,s)+L_{2}^{H}(s,t)+L_{2}^{H}(t,s)\big)\Big) \phi_1(t)\phi_2(s)\,dt \,ds <\infty. 
$$
Thus we conclude the proof using the dominated convergence theorem.

\end{proof} 
Before we prove Corollary \ref{corollary-gmc}, we recall an approximation of $\log_{+}(1/|x|)$, from Example 2.2 in \cite{Rhodes-Vargas14}.
For any $x\in \re$ we define the cone $\mathcal C(x)$ in $\re\times \re_{+}$: 
$$
\mathcal C(x) = \Big\{ (y,t) \in \re\times \re_{+} ; \, |x-y| \leq  \frac{t\wedge 1}{2} \Big \}. 
$$
A direct computation gives 
$$
f(x) =\int_{\mathcal C(0)\cap \mathcal C(x) } \frac{dy\,dt}{t^{2}} = \log_{+}\frac{1}{|x|}.
$$
We approximate $f$ by the following functions,
%\bn 
%f_{n}(x) = \int_{\mathcal C(0)\cap \mathcal C(x); \, 1/n \leq t <1/(n-1)}\frac{dy\,dt}{t^{2}}, \quad %n=1,2,..., 
%\en
%where for $n=1$ the integration is carried over the set $\{\mathcal C(0)\cap \mathcal C(x);1 \leq t <\infty\}$. 
%It follows that 
%$$
%f(x)  = \sum_{n\geq 1} f_{n}(x). 
%$$
%We also define
\be \label{log-approx} 
\tilde f_{\eps}(x) := \int_{\mathcal C(0)\cap \mathcal C(x); \, \eps < t <\infty }\frac{dy\,dt}{t^{2}}, \quad \textrm{for } \eps>0.
\ee
\paragraph{Proof of Corollary \ref{corollary-gmc}}
From Theorem 25 in \cite{shamov16} we obtain that Corollary \ref{corollary-gmc} follows from (\ref{ker-L-0}), provided we show that $\{\xi^{H}\}_{H\in (0,1)}$ are uniformly integrable.\medskip \\ 
First note that all $\xi^{H}$, $0<H<1$, can be constructed on the same probability space, as a convolution of the same Brownian motion with different kernels which depend on $H$ (see \cite{Man-Van68}). 
From (\ref{KH-deomp}) and (\ref{I24-bnd}) we have for $\dl \leq s,t\leq 1$,  
\bn
 K_{H}(s,t) &\leq & \frac{1}{2H}|1-|t-s|^{2H}|+  \frac{1}{2H}|s^{2H}-1| +   \frac{1}{2H}|t^{2H}-1|+(t-s)^{2H} + C \\
 &\leq &  \frac{1}{2H}|1-|t-s|^{2H}| +C(\dl). 
 \en
Using $1-e^{-x} \leq x$ for $x>0$ we get for all $\dl \leq s,t\leq 1$, 
\bn
 K_{H}(s,t) &\leq & -\log|t-s| + C(\dl). 
\en
On the other hand for $\dl \leq s,t\leq 1$ we have 
\bn
 K_{H}(s,t) &\leq & \frac{1}{2H}+ C(\dl). 
\en
It follows that 
\bn
 K_{H}(s,t) &\leq & \big(\frac{1}{2H}\big) \wedge \log\frac{1}{|t-s|} + C(\dl) \\
 &= & \log\frac{1}{|t-s| \vee e^{-\frac{1}{2H}}} + C(\dl). 
\en
Let $\eps \in(0,1)$ be arbitrary small. Let $\wt X^{\eps}$ be a centred Gaussian random field with the covariance kernel $\tilde f_{\eps}+C(\dl)$, where $\tilde f_{\eps}$ is given in (\ref{log-approx}), i.e. 
$$
\wt K_{\eps}(t,s) := \mathbb{E}[\wt X^{\eps}_{t} \wt X^{\eps}_{s}] = \tilde f_{\eps}(t-s) +C(\dl), \quad -\infty <t,s <\infty. 
$$

Since $\eps \in(0,1)$, we get from (\ref{log-approx}) for $0< x \leq1$,
\bn
 \tilde f_{\eps}(x)&=&\int_{\mathcal C(0)\cap \mathcal C(x); \, \eps < r <\infty }\frac{dy\,dr}{r^{2}} \\
 &=&\int_{x\vee \eps}^{\infty}\frac{dr}{r^{2}}\int^{(r\wedge 1)/2}_{x-(r\wedge 1)/2}dy \\
& =& -\log(x\vee \eps) -\frac{x}{x\vee\eps} +1. 
\en
Note that $\tilde f_{\eps}$ is an even function and therefore,  
$$
\wt K_{\eps}(t,s) \geq \log\frac{1}{|t-s|\vee \eps} +C(\dl), \quad 0< t,s \leq1. 
$$ 
We define the following of measure 
\bn
\wt \xi_{\gamma}^{\eps}(dt) = e^{\gamma \wt X^{\eps}_{t}-\frac{\gamma^{2} }{2}\mathbb{E}[(\wt X_{t}^{\eps})^{2}]}dt,  \quad \dl \leq t\leq 1. 
\en 
Recall that $\gamma^{2}< 2$ by our assumption. Then from the proof of Proposition 3.5 in \cite{Robert-Vargas10} it follows that there exists $p=p(\gamma) >1$, such that 
$$
\sup_{\eps>0}\mathbb E\big[\big(\wt \xi_{\gamma}^{\eps}([0,1])\big)^{p}\big] <\infty. 
$$
By chosing $H=H(\eps) = (-2\log \eps)^{-1}$, we note that $K_{H}(s,t) \leq \wt K_{\eps}(s,t)$, and  from the comparison principal with $F(x) =x^{p}$ (see Corollary A.2 in \cite{Robert-Vargas10}), we get 
$$
\sup_{H\in(0,1)}\mathbb E\big[\xi_{\gamma}^{H}([\dl,1])^{p} \big] <\infty. 
$$
We therefore conclude that $\{\xi^{H}\}_{H\in (0,1)}$ are uniformly integrable. 
\qed 

\section*{Acknowledgments}
We are very grateful to Vincent Vargas, an anonymous referee and the associate editor
whose numerous useful comments enabled us to significantly improve this paper.

%\bibliographystyle{plain}
%\printindex
%\bibliography{FBM-b_271017}

\bigskip
%\noindent Eyal Neuman: Department of Mathematics, University of Rochester, Rochester, 14627 NY, USA. eneuman4@ur.rochester.edu

%\medskip
%\noindent 
\end{document}